\documentclass[10pt]{amsart}
\usepackage[a4paper,marginratio=1:1]{geometry}
\usepackage[initials,non-sorted-cites]{amsrefs}
\usepackage{tensor,braket,bm}

\newtheorem{thm}{Theorem}[section]
\newtheorem*{thmodd}{Theorem 0.1$'$}
\newtheorem{prop}[thm]{Proposition}
\newtheorem{lem}[thm]{Lemma}

\theoremstyle{definition}

\theoremstyle{remark}
\newtheorem{rem}[thm]{Remark}

\newcommand{\abs}[1]{\lvert#1\rvert}
\newcommand{\bdry}{\partial}

\DeclareMathOperator{\Ric}{Ric}
\DeclareMathOperator{\Scal}{Scal}
\DeclareMathOperator{\tr}{tr}

\DeclareMathOperator{\grad}{grad}
\DeclareMathOperator{\im}{im}

\numberwithin{equation}{section}

\title{$Q$-curvature of Weyl structures and Poincar\'e metrics}
\author{Kengo Hirachi}
\address{Graduate School of Mathematical Sciences, The University of Tokyo,
	3-8-1 Komaba, Meguro, Tokyo 153-8914, Japan}
\email{hirachi@ms.u-tokyo.ac.jp}
\author{Christian L\"ubbe}
\address{Department of Economics, Mathematics and Statistics, Birkbeck, University of London,
	Malet Place, London WC1E 7HX, United Kingdom}
\email{c.luebbe@bbk.ac.uk}
\author{Yoshihiko Matsumoto}
\address{Department of Mathematics, Tokyo Institute of Technology,
	2-12-1 Ookayama, Meguro, Tokyo 152-8551, Japan}
\email{matsumoto@math.titech.ac.jp}

\begin{document}

\begin{abstract}
	We study an asymptotic Dirichlet problem for Weyl structures on asymptotically hyperbolic manifolds.
	By the bulk-boundary correspondence, or more precisely by the Fefferman--Graham theorem on Poincar\'e metrics,
	this leads to a natural extension of the notion of Branson's $Q$-curvature to Weyl structures
	on even-dimensional conformal manifolds.
\end{abstract}

\maketitle

\section*{Introduction}

Let $\overline{X}=X\sqcup\bdry X$ be a smooth compact manifold-with-boundary of dimension $n+1$,
and $g$ a smooth conformally compact metric on $X$, i.e., a Riemannian metric for which
$r^2g$ extends to a smooth metric $\overline{g}$ on $\overline{X}$, where
$r\in C^\infty(\overline{X})$ is any boundary defining function.
The metric $g$ is called \emph{asymptotically hyperbolic} (abbreviated as AH) if it moreover satisfies
$\abs{dr}_{\overline{g}}=1$ on $\bdry X$.
Such a pair $(X,g)$ is a generalization of the ball model of the hyperbolic space $\mathbb{H}^{n+1}$.
The \emph{conformal infinity} of $(X,g)$ is the boundary $M=\bdry X$ equipped with the conformal class
$\mathcal{C}$ determined by $\overline{g}|_{TM}$, which is independent of $r$.

In this article, we introduce the notion of the $Q$-curvature of Weyl structures on $(M,\mathcal{C})$
through studying a Dirichlet-type problem for Weyl structures on $(\overline{X},\overline{\mathcal{C}})$,
where $\overline{\mathcal{C}}$ is the conformal class of $\overline{g}$.
Our work is a generalization of Fefferman--Graham's characterization~\cite{Fefferman_Graham_02}
of Branson's $Q$-curvature~\cite{Branson_95}.

By definition, a \emph{Weyl structure} (or a \emph{Weyl connection}) $\nabla$ on $(M,\mathcal{C})$ is
a torsion-free linear connection on $M$ that preserves the class $\mathcal{C}$.
If we pick any representative metric $h\in\mathcal{C}$ as a ``reference metric'' and let $\nabla^h$ be
the associated Levi-Civita connection,
then a torsion-free linear connection $\nabla$ is a Weyl structure if and only if it satisfies
$\nabla=\nabla^h+\beta$ for some (unique) 1-form $\beta\in\Omega^1(M)$, meaning
$\nabla h=-2\beta\otimes h$, or equivalently
\begin{equation*}
	\nabla_\xi\eta=\nabla^h_\xi\eta+\beta(\xi)\eta+\beta(\eta)\xi-h(\xi,\eta)\beta^\sharp,
\end{equation*}
where $\beta^\sharp$ is the metric dual of $\beta$.
If $h'=e^{2\Upsilon}h\in\mathcal{C}$ is another representative, where $\Upsilon\in C^\infty(M)$, then
the 1-form $\beta'$ satisfying $\nabla=\nabla^{h'}+\beta'$ is given by $\beta'=\beta-d\Upsilon$.
Therefore, a Weyl structure $\nabla=\nabla^h+\beta$ is a Levi-Civita connection if and only if $\beta$
is exact, and is locally a Levi-Civita connection if and only if $\beta$ is closed.
In the latter case, we also say that $\nabla$ itself is \emph{closed}.

Suppose $(X,g)$ is given, and let $\overline{\nabla}$ be a Weyl structure
on $(\overline{X},\overline{\mathcal{C}})$.
As $\overline{\nabla}$ may not be a Levi-Civita connection, its curvature tensor does not necessarily satisfy the
usual Riemannian symmetry properties.
In particular, the Ricci tensor is not symmetric in general.
We call the skew-symmetric part of $\Ric_{\overline{\nabla}}$ the \emph{Faraday tensor} $F_{\overline{\nabla}}$.
It is known that, if $\overline{g}\in\overline{\mathcal{C}}$ is any representative and
$\overline{\nabla}=\nabla^{\overline{g}}+\overline{b}$, then $F_{\overline{\nabla}}$
equals a constant times $d\overline{b}$ (the constant being dependent on convention).
Consequently, the Faraday tensor $F_{\overline{\nabla}}$ determines $\overline{\nabla}$
up to addition of a closed 1-form.

We consider the following curvature constraint, which is the Euler--Lagrange equation
for the Lagrangian density $\abs{F_{\overline{\nabla}}}_g^2$:
\begin{equation}
	\label{eq:divergence_free_Faraday}
	d_g^*F_{\overline{\nabla}}=0.
\end{equation}
We have a canonical reference metric for $\overline{\nabla}$ on $X$, which is the metric $g$.
By putting $\overline{\nabla}=\nabla^g+b$, we can reformulate \eqref{eq:divergence_free_Faraday} into
an equation for a 1-form $b\in\Omega^1(X)$, which is known as the (massless) Proca equation:
\begin{equation}
	\label{eq:Proca}
	d_g^*db=0.
\end{equation}
Since $F_{\overline{\nabla}}$ is invariant under the change
$\overline{\nabla}\rightsquigarrow\overline{\nabla}+\gamma$ for any closed 1-form $\gamma\in\Omega^1(\overline{X})$,
so is equation \eqref{eq:Proca}.
To break this gauge invariance as much as possible, we introduce the Feynman gauge condition:
\begin{equation}
	\label{eq:Feynman}
	d_g^*b=0.
\end{equation}
Then clearly, the solutions of the system of equations \eqref{eq:Proca} and \eqref{eq:Feynman} have only
the freedom of adding harmonic 1-forms.

The natural Dirichlet data for Weyl structures $\overline{\nabla}$ on $(\overline{X},\overline{\mathcal{C}})$
are given by those on $(M,\mathcal{C})$;
note that the notion of the induced Weyl structure on $M$ by $\overline{\nabla}$ makes sense because
$\overline{\mathcal{C}}$ determines the orthogonal decomposition $(T\overline{X})|_M=TM\oplus T^\perp M$.
The Dirichlet problem for our system of equations can be solved as follows.

\begin{thm}
	\label{thm:existence_extension}
	Let $n$ be even and $n\ge 4$.
	Suppose that $g$ is an AH smooth conformally compact metric on $X$,
	and let $\nabla$ be a smooth Weyl structure on the conformal infinity $(M,\mathcal{C})$,
	where $M=\bdry X$.
	Then there exists a $C^{n-3}$ Weyl structure $\overline{\nabla}$ on $\overline{X}$
	with induced Weyl structure $\nabla$ on $M$ satisfying \eqref{eq:Proca} and \eqref{eq:Feynman}.
	It is unique up to addition of an $L^2$-harmonic 1-form on $X$.
\end{thm}

It is known that any $L^2$-harmonic 1-form $\gamma\in\Omega^1(X)$ is smoothly extended to $\overline{X}$,
which is a consequence of the fact that $\gamma$ admits a ``polyhomogeneous expansion'' and
its logarithmic term coefficients all vanish since $\gamma|_{TM}=0$
(see~Proposition \ref{prop:harmonic_extension_of_1-forms} and \cite{Aubry_Guillarmou_11}*{Section 3.1.1}).
Therefore, adding $L^2$-harmonic 1-forms does not break the $C^{n-3}$ boundary regularity
of $\overline{\nabla}$.

We made an assumption on $n$ in the theorem above because this is the case of our main interest.
However, the following theorem for $n\ge 3$ odd can be proved almost by the same argument.
Again, $L^2$-harmonic 1-forms are smooth up to the boundary.

\begin{thmodd}
	Let $n$ be odd and $n\ge 3$, and $(X,g)$, $\nabla$ as in Theorem \ref{thm:existence_extension}.
	Then there exists a smooth Weyl structure $\overline{\nabla}$ on $\overline{X}$ with
	induced Weyl structure $\nabla$ on $M$ satisfying \eqref{eq:Proca} and \eqref{eq:Feynman}.
	It is unique up to addition of an $L^2$-harmonic 1-form on $X$.
\end{thmodd}

We do not have similar results for $n=1$, $2$ because Mazzeo's work~\cite{Mazzeo_88}, which gives the
analytic basis to our argument, does not apply in these dimensions.

Now let $n$ be even and $n\ge 4$.
We next focus on the obstruction to the smoothness of $\overline{\nabla}$
to get a quantity that is conformally invariantly assigned to $\nabla$,
as Graham and Zworski~\cite{Graham_Zworski_03} did for functions to characterize
the GJMS operators~\cite{Graham_Jenne_Mason_Sparling_92}.
For our purpose, $g$ should be canonically determined to a sufficient order
only by the conformal class $\mathcal{C}$.
Hence we take the \emph{Poincar\'e metric} of Fefferman--Graham~\cite{Fefferman_Graham_85, Fefferman_Graham_12},
which satisfies
\begin{equation*}
	\Ric_g=-ng+O(r^n)\qquad\text{and}\qquad \tr_g(\Ric_g+ng)=O(r^{n+2})\qquad\text{at $\bdry X$}.
\end{equation*}
(The first condition means that $\abs{\Ric(g)+ng}_g=O(r^n)$.)
If $\mathcal{C}$ is given, then such a $g$ exists, and is unique up to an $O(r^n)$ error
with $O(r^{n+2})$ trace and the action of diffeomorphisms of
$\overline{X}$ that restricts to the identity on $\bdry X$.
Then the aforementioned obstruction is determined only by the pair $(\mathcal{C},\nabla)$.
Furthermore, it turns out that it is naturally interpreted as a tractor on $M$.
Let us set up the notation: $\mathcal{E}[w]$ is the density bundle of conformal weight $w$ over $M$,
$\mathcal{S}$ is the standard conformal tractor bundle, $\mathcal{S}[w]=\mathcal{S}\otimes\mathcal{E}[w]$, and
$\mathcal{S}^*[w]=\mathcal{S}^*\otimes\mathcal{E}[w]$.
For the definition of these bundles, we refer to Bailey--Eastwood--Gover~\cite{Bailey_Eastwood_Gover_94}
or Eastwood's expository article~\cite{Eastwood_96}.
By abuse of notation, the spaces of smooth sections of these bundles are denoted by the same symbols.
Then we have the following.

\begin{thm}
	\label{thm:smoothness_extension}
	Let $g$ be the Poincar\'e metric on $X$, and $\nabla$ a smooth Weyl structure on $(M,\mathcal{C})$.
	Then there exists a density-weighted standard cotractor $\bm{Q}_\nabla\in\mathcal{S}^*[n+1]$ on $M$,
	which is locally determined by $(\mathcal{C},\nabla)$, such that
	any $C^{n-3}$ extension $\overline{\nabla}$ in Theorem \ref{thm:existence_extension}
	is smooth if and only if $\bm{Q}_\nabla$ vanishes.
\end{thm}

Let $h\in\mathcal{C}$ and $\beta\in\Omega^1(M)$ be such that $\nabla=\nabla^h+\beta$.
The choice of $h$ determines a direct sum decomposition
$\mathcal{S}^*\cong\mathcal{E}[-1]\oplus\Omega^1[1]\oplus\mathcal{E}[1]$, where
$\Omega^1[1]=\Omega^1(M)\otimes\mathcal{E}[1]$.
Via this decomposition and the trivialization of the density bundles by $h$, the tractor $\bm{Q}_\nabla$ is given by
\begin{equation}
	\label{eq:Q-tractor}
	\bm{Q}_\nabla\overset{h}{=}
	(-1)^{n/2-1}2^{n-2}(n/2-1)!^2
	\begin{pmatrix}
		Q_01+G_1\beta & L_1\beta & 0
	\end{pmatrix}.
\end{equation}
Here we used the Branson--Gover operators~\cite{Branson_Gover_05}
$L_1\colon\Omega^1(M)\to\Omega^1(M)$, $G_1\colon\Omega^1(M)\to C^\infty(M)$ and
$Q_0\colon C^\infty(M)\to C^\infty(M)$
(adopting the normalization of Aubry--Guillarmou~\cite{Aubry_Guillarmou_11}). In particular,
\begin{equation*}
	Q_01=\frac{(-1)^{n/2-1}}{2^{n-2}(n/2-1)!^2}Q_h,
\end{equation*}
where $Q_h$ is Branson's $Q$-curvature of $h$.
Since it is known that $L_1$ and $G_1$ annihilate closed forms (see \cite{Branson_Gover_05}),
$\bm{Q}_\nabla$ is essentially Branson's $Q$-curvature when $\nabla$ is a Levi-Civita connection.
The authors propose to call $\bm{Q}_\nabla$ the \emph{$Q$-curvature tractor} of the Weyl structure $\nabla$.

For given $\nabla$, we consider the natural pairing of $\bm{Q}_\nabla$ and
another canonical tractor $\bm{W}_\nabla\in\mathcal{S}[-1]$ associated to $\nabla$.
By using any metric $h\in\mathcal{C}$ and $\beta\in\Omega^1(M)$ for which $\nabla=\nabla^h+\beta$,
we define
\begin{equation*}
	\bm{W}_\nabla\overset{h}{=}\begin{pmatrix}
		1 \\ -\beta^\sharp \\ \frac{1}{2}\abs{\beta}^2
	\end{pmatrix}.
\end{equation*}
Then the pairing $Q_\nabla=\braket{\bm{Q}_\nabla,\bm{W}_\nabla}\in\mathcal{E}[n]$ can be integrated. Since
\begin{equation*}
	Q_\nabla/dV_h=Q_h+(-1)^{n/2-1}2^{n-2}(n/2-1)!^2(G_1\beta-\braket{L_1\beta,\beta}),
\end{equation*}
we may use the fact that $G_1\beta$ is the divergence of some 1-form to conclude that, for $M$ compact,
$Q_\nabla$ integrates to the following global invariant of $(M,\mathcal{C},\nabla)$:
\begin{equation}
	\label{eq:integral}
	\int_M Q_hdV_h+(-1)^{n/2}2^{n-2}(n/2-1)!^2\int_M \braket{L_1\beta,\beta}dV_h.
\end{equation}
This can be seen as a functional in the space of Weyl structures on $(M,\mathcal{C})$.
As the first term, the total $Q$-curvature, is an invariant of $\mathcal{C}$,
the formula above makes us curious about the spectrum of $L_1$.
There are explicit formulae for $n=4$ and $6$~\cite{Aubry_Guillarmou_11}*{Section 8}:
\begin{equation*}
	L_1=\frac{1}{2}d^*d\quad\text{($n=4$)},\qquad
	L_1=-\frac{1}{16}d^*\left(\Delta_h-\Ric+\frac{2}{5}\Scal\right)d\quad\text{($n=6$)}.
\end{equation*}
Here $\Ric$ acts as an endomorphism.
In four dimensions, this implies that the second term in \eqref{eq:integral} is nonnegative and vanishes
if and only if $\beta$, or equivalently $\nabla$, is closed.
Hence the integral of $Q_\nabla$ minimizes at closed Weyl structures.
The same is true in six dimensions under some assumption on the Ricci tensor.
In general dimensions, a formula of $L_1$ can be obtained for an Einstein metric $h$
by using the idea in third author's article~\cite{Matsumoto_13}.
If $\Ric_h=2\lambda(n-1)h$ so that the Schouten tensor is $P_h=\lambda h$,
\begin{equation}
	\label{eq:explicit_formula_for_conformally_Einstein}
	L_1=\frac{(-1)^{n/2}}{2^{n-3}(n/2-1)!(n/2-2)!}d^*\left(\prod_{m=1}^{n/2-2}(\Delta_h-2m(m-n+3)\lambda)\right)d.
\end{equation}
One may conclude by this that, if $\mathcal{C}$ contains an Einstein metric with positive scalar curvature,
then the integral of $Q_\nabla$ minimizes exactly at Levi-Civita connections
(note that Bochner's Theorem assures the vanishing of $H^1(M)$).

Our theorems are applications of the previous results on the Dirichlet problems for functions and
differential forms on AH manifolds. The analytic aspect is due to
Mazzeo--Melrose~\cite{Mazzeo_Melrose_87} and Mazzeo~\cite{Mazzeo_88},
while the asymptotic expansions were investigated thoroughly by
Graham--Zworski~\cite{Graham_Zworski_03} and Aubry--Guillarmou~\cite{Aubry_Guillarmou_11}.
A direct connection to Branson's $Q$-curvature was found by Fefferman--Graham~\cite{Fefferman_Graham_02}.
In Section~\ref{sec:functions_and_1-forms}, we recall their results that are necessary here.
We prove our main theorems in Section~\ref{sec:Weyl_connection},
and the proof of \eqref{eq:explicit_formula_for_conformally_Einstein} is given in
Section \ref{sec:conformally_Einstein}.
(For our analysis of $\bm{Q}_\nabla$, formal asymptotic expansions suffice our needs
and the deep results of~\cite{Mazzeo_Melrose_87,Mazzeo_88} are not really necessary.
However we choose to use them for a clearer exposition.)
We shall concentrate on the case where $n$ is even and leave the proof of Theorem 0.1$'$ to the interested reader.

\subsection*{Acknowledgments}

This work started during CL's visit to the University of Tokyo in 2014 and
the preparation of the manuscript was finished during YM's visit to
the \'Ecole normale sup\'erieure in Paris in 2014--15.
They would like to acknowledge the kind hospitality of the both institutions.
KH is partially supported by JSPS KAKENHI grant 60218790.
CL is partially supported by JSPS Postdoctoral Fellowship for North American and European Researchers
(Short-term) PE 13079.
YM is partially supported by JSPS Postdoctoral Fellowship and KAKENHI grant 26-11754.

\section{Preliminaries: Dirichlet problem for functions and 1-forms}
\label{sec:functions_and_1-forms}

We always assume that $n$ is even and $n\ge 4$ in the sequel.
Let $g$ be an AH smooth conformally compact metric on $X$.
It is well known~\cite{Graham_Lee_91}*{Section 5} that a sufficiently small open neighborhood $\mathcal{U}$ of
$M\subset\overline{X}$ can be identified with the product $M\times[0,\varepsilon)$ so that
\begin{equation}
	\label{eq:normalization}
	g=\frac{dx^2+h_x}{x^2},
\end{equation}
where $x$ is the coordinate on the second factor of $M\times[0,\varepsilon)$ and
$h_x$ is a smooth 1-parameter family of Riemannian metrics on $M$.
The metric $h=h_0$ is a representative of the conformal class $\mathcal{C}$.
In fact, for any prescribed $h\in\mathcal{C}$, there is such an identification;
moreover, $h$ determines the identification near $\bdry X$.
We call the expression \eqref{eq:normalization} the \emph{normalization} of $g$,
and $x$ the \emph{normalizing boundary defining function} of $\overline{X}$, with respect to $h$.

We shall summarize fundamental results on the Dirichlet problems for functions and 1-forms.
In the original papers, some of them are stated under (weak or genuine) Einstein conditions,
but they are actually valid in the following general setting.
Asymptotic expansions in the propositions below are given with respect to the identification
$\mathcal{U}\cong M\times[0,\varepsilon)$ associated to some fixed $h$.

\begin{prop}[Mazzeo--Melrose~\cite{Mazzeo_Melrose_87}, Graham--Zworski~\cite{Graham_Zworski_03}]
	\label{prop:harmonic_extension_of_functions}
	For any function $\varphi\in C^\infty(M)$,
	there exists a unique harmonic function $\overline{f}\in C^{n-1}(\overline{X})$ with boundary value $\varphi$.
	It has the following expansion at the boundary:
	\begin{equation*}
		\overline{f}=\varphi+\sum_{k=1}^{n-1}x^k\varphi_k+x^n\log x\cdot L_0\varphi+O(x^n),
		\qquad\varphi_k\in C^\infty(M).
	\end{equation*}
	Here $L_0$ is a linear differential operator locally determined by $g$ and $h$,
	and $\overline{f}$ is smooth if $L_0\varphi$ vanishes.
	If $g$ is the Poincar\'e metric, $L_0$ is the GJMS operator of critical order up to normalization.
\end{prop}

The solvability of the Dirichlet problem and the appearance of the first logarithmic term at the power $x^n$ are
consequences of the fact that the characteristic exponents of the Laplacian on functions are
$0$ and $n$: $\Delta_g$ on functions is expressed as
\begin{equation}
	\label{eq:Laplacian_on_functions}
	\Delta_g=-(x\partial_x)^2+nx\partial_x+xR,
\end{equation}
in which $R$ is a polynomial of vector fields that are tangent to $\bdry X$.

A similar technique was used to obtain
the following ``direct'' characterization of Branson's $Q$-curvature in terms of the Poincar\'e metric.

\begin{prop}[Fefferman--Graham~\cite{Fefferman_Graham_02}]
	\label{prop:harmonic_defining_function}
	For any representative metric $h\in\mathcal{C}$ and the associated normalizing boundary defining function $x$,
	there exists a unique function $\rho$ such that $u=\log\rho-\log x\in C^{n-1}(\overline{X})$,
	$\log\rho$ is harmonic, and $u|_{\bdry X}=0$. It has the following expansion:
	\begin{equation*}
		\log\rho=\log x+\sum_{k=1}^{n-1}x^kr_k+x^n\log x\cdot s+O(x^n),
		\qquad r_k,\ s\in C^\infty(M).
	\end{equation*}
	The function $u$ is smooth if $s$ vanishes. If $g$ is the Poincar\'e metric, then
	\begin{equation*}
		s=\frac{(-1)^{n/2-1}}{2^{n-1}(n/2)!(n/2-1)!}Q_h,
	\end{equation*}
	where $Q_h$ is Branson's $Q$-curvature of $h$.
\end{prop}

The corresponding problem for differential forms is studied in~\cite{Mazzeo_88, Aubry_Guillarmou_11}.
Though differential forms of general degrees are considered in these works, we only use the 1-form case.
For a later need, we state the result for general inhomogeneous equations, which also follows from their approach.

\begin{prop}[Mazzeo~\cite{Mazzeo_88}, Aubry--Guillarmou~\cite{Aubry_Guillarmou_11}]
	\label{prop:harmonic_extension_of_1-forms}
	Let $\overline{a}\in\Omega^1(\overline{X})$ be a smooth 1-form on $\overline{X}$
	such that $\overline{a}|_{TM}=0$.
	Then for any 1-form $\beta\in\Omega^1(M)$,
	there exists a solution $\overline{b}\in C^{n-3}(\overline{X},T^*\overline{X})$ to the equation
	$\Delta_g\overline{b}=\overline{a}$ satisfying $\overline{b}|_{TM}=\beta$,
	which is unique modulo $L^2$-harmonic 1-forms. It allows the expansion
	\begin{equation*}
		\overline{b}=\beta+\sum_{k=1}^{n-3}x^k\beta_k+x^{n-2}\log x\cdot \beta^{(1)}
		+\left(\sum_{k=0}^{n-2}x^k\varphi_k+x^{n-1}\log x\cdot\varphi^{(1)}\right)dx+O^+(x^{n-2}),
	\end{equation*}
	where $\beta_k$, $\beta^{(1)}\in\Omega^1(M)$, $\varphi_k$, $\varphi^{(1)}\in C^\infty(M)$
	and the remainder $O^+(x^{n-2})$ is an $O(x^{n-2})$ term that becomes $O(x^{n-1})$
	when contracted with $\partial_x$.
	The solution $\overline{b}$ is smooth if $\beta^{(1)}$ and $\varphi^{(1)}$ both vanish.

	If $\overline{a}=0$, then there are linear differential operators $L_1$ and $G_1$
	locally determined by $g$ and $h$ for which $\beta^{(1)}=L_1\beta$, $\varphi^{(1)}=G_1\beta$.
	Moreover, if $g$ is the Poincar\'e metric,
	then $L_1$ and $G_1$ are the Branson--Gover operators up to normalization.
\end{prop}

\section{Proof of main theorems}
\label{sec:Weyl_connection}

Let $\nabla$ be a Weyl structure on $(M,\mathcal{C})$.
As explained in Introduction, the construction of the extension $\overline{\nabla}$ in
Theorem \ref{thm:existence_extension} boils down to a Dirichlet problem on 1-forms.
However, in order to apply Proposition~\ref{prop:harmonic_extension_of_1-forms} for this purpose,
$g$ is not appropriate as a reference metric for $\overline{\nabla}$.
Indeed, since $g$ diverges at $\bdry X$, so does the 1-form $b$ satisfying $\overline{\nabla}=\nabla^g+b$.

A good choice of reference metric is $\overline{g}=\rho^2g$, where $\rho$ is the function given in
Proposition \ref{prop:harmonic_defining_function} for some $h\in\mathcal{C}$.
Since $\rho$ is a (possibly non-smooth) defining function,
$\overline{g}$ is a metric on $\overline{X}$ that represents $\overline{\mathcal{C}}$.
If we take the 1-form $\overline{b}$ for which $\overline{\nabla}=\nabla^{\overline{g}}+\overline{b}$,
then since $\overline{b}=b-d\log\rho$ and $\Delta_g\log\rho=0$,
\eqref{eq:Proca} and \eqref{eq:Feynman} are equivalent to $d_g^*d\overline{b}=0$ and $d_g^*\overline{b}=0$.
Obviously, for this system to be satisfied, it is necessary that
\begin{equation}
	\Delta_g\overline{b}=0.
\end{equation}
The converse holds actually. In fact, if $\Delta_g\overline{b}=0$ then $\Delta_g(d_g^*\overline{b})=0$ follows.
By the conformal change law of the divergence (see Besse~\cite{Besse_87}*{1.159 Theorem}),
$d_g^*\overline{b}=\rho^2d_{\overline{g}}^*\overline{b}+(n-1)\rho\braket{d\rho,\overline{b}}_{\overline{g}}$
is continuous up to the boundary and vanishes on $\bdry X$,
so the maximum principle implies that $d_g^*\overline{b}=0$.
Hence we also have $d_g^*d\overline{b}=0$.

\begin{proof}[Proof of Theorem~\ref{thm:existence_extension}]
	Take an arbitrary pair $(h,\beta)$ so that $\nabla=\nabla^h+\beta$.
	We define $\overline{g}=\rho^2g$, where $\rho$ is the function
	in Proposition \ref{prop:harmonic_defining_function} associated to $h$.
	Then by Proposition \ref{prop:harmonic_extension_of_1-forms},
	there is a 1-form $\overline{b}\in C^{n-3}(\overline{X};T^*\overline{X})$ such that
	$\Delta_g\overline{b}=0$ and $\overline{b}|_{TM}=\beta$.
	We set
	\begin{equation*}
		\overline{\nabla}=\nabla^{\overline{g}}+\overline{b}.
	\end{equation*}
	Then \eqref{eq:Proca} and \eqref{eq:Feynman} follow because $\Delta_g\overline{b}=0$ holds.
	Moreover, for any vector fields $\xi$, $\eta\in\mathfrak{X}(\overline{X})$ that are tangent to $\bdry X$,
	the tangential component of $\overline{\nabla}_\xi\eta$ is
	$\nabla^h_\xi\eta+\beta(\eta)\xi+\beta(\xi)\eta-h(\xi,\eta)\beta^\sharp$, which is $\nabla_\xi\eta$.
	In this construction, there is an ambiguity in $\overline{b}$ that lies in the $L^2$-kernel of
	$\Delta_g$ on 1-forms.
	Since $\overline{b}|_{TM}=\beta$ is necessary in order that $\overline{\nabla}$ induces $\nabla$,
	there is no other ambiguities.
\end{proof}

It is interesting to see directly that another choice $(h',\beta')$ would lead to the
same Weyl structure $\overline{\nabla}$ (modulo, of course, $L^2$-harmonic 1-forms).
If $\nabla=\nabla^h+\beta=\nabla^{h'}+\beta'$,
then we can write $h'=e^{2\Upsilon}h$ and $\beta'=\beta-d\Upsilon$ by some $\Upsilon\in C^\infty(M)$.
Let $\overline{\Upsilon}$ be the harmonic extension of $\Upsilon$, which uniquely exists by
Proposition~\ref{prop:harmonic_extension_of_functions}.
Then the function $\rho'$ in Proposition \ref{prop:harmonic_defining_function} associated to $h'$ is
$\rho'=e^{\overline{\Upsilon}}\rho$, and hence
$\smash{\overline{g}}'=\smash{\rho'}^2g=e^{2\overline{\Upsilon}}\overline{g}$.
On the other hand, a solution to $\Delta_g\smash{\overline{b}}'=0$ and
$\smash{\overline{b}}'|_{TM}=\beta'$ is given by $\smash{\overline{b}}'=\overline{b}-d\overline{\Upsilon}$.
Therefore, $\nabla^{\smash{\overline{g}}'}+\smash{\overline{b}}'$ and
$\nabla^{\overline{g}}+\overline{b}$ are the same.

Next we discuss the smoothness issue.

\begin{lem}
	\label{lem:smoothness}
	Let $h\in\mathcal{C}$ and $\beta\in\Omega^1(M)$ be such that $\nabla=\nabla^h+\beta$.
	Then, the Weyl structure $\overline{\nabla}$ in Theorem \ref{thm:existence_extension} is smooth
	if and only if
	\begin{equation}
		\label{eq:vanishing_first_log_terms}
		L_1\beta=0\qquad\text{and}\qquad ns+G_1\beta=0,
	\end{equation}
	where $s\in C^\infty(M)$ is given in Proposition \ref{prop:harmonic_defining_function}
	and $L_1$, $G_1$ are as in Proposition \ref{prop:harmonic_extension_of_1-forms}.
\end{lem}

\begin{proof}
	We take the normalization of the metric $g$ with respect to $h$,
	and take the 1-form $\tilde{b}$ so that $\overline{\nabla}=\nabla^{x^2g}+\tilde{b}$.
	Then, since $x^2g$ is smooth up to $\bdry X$,
	$\overline{\nabla}$ is smooth if and only if $\tilde{b}$ is smooth.
	Using $\rho$ and $\overline{b}$ constructed in the proof of Theorem \ref{thm:existence_extension},
	$\tilde{b}$ is computed as follows:
	\begin{equation*}
		\begin{split}
			\tilde{b}&=(d\log\rho+\overline{b})-d\log x\\
			&=d\left(\sum_{k=1}^{n-1}x^kr_k+x^n\log x\cdot s\right)\\
			&\phantom{\;=\;}
			+\beta+\sum_{k=1}^{n-3}x^k\beta_k+x^{n-2}\log x\cdot L_1\beta
			+\left(\sum_{k=0}^{n-2}x^k\varphi_k+x^{n-1}\log x\cdot G_1\beta\right)dx+O^+(x^{n-2})\\
			&=\beta+\sum_{k=1}^{n-3}x^k(\beta_k+dr_k)+x^{n-2}\log x\cdot L_1\beta\\
			&\phantom{\;=\;}
			+\left(\sum_{k=0}^{n-2}x^k(\varphi_k+(k+1)r_{k+1})+x^{n-1}\log x\cdot(ns+G_1\beta)\right)dx
			+O^+(x^{n-2}).
		\end{split}
	\end{equation*}
	Therefore, \eqref{eq:vanishing_first_log_terms} is equivalent to that the first logarithmic terms of
	$\tilde{b}$ being zero; thus \eqref{eq:vanishing_first_log_terms} is necessary for the smoothness.
	Furthermore, since $\Delta_g\tilde{b}=-\Delta_gd\log x=-d\Delta_g\log x$ and
	$\Delta_g\log x\in xC^\infty(\overline{X})$ by an explicit computation,
	it follows that $(\Delta_g\tilde{b})|_{TM}=0$.
	Hence by Proposition \ref{prop:harmonic_extension_of_1-forms}, \eqref{eq:vanishing_first_log_terms} is
	also sufficient.
\end{proof}

Let us specialize to the case where $g$ is the Poincar\'e metric.
Then, since $ns=Q_01$, \eqref{eq:vanishing_first_log_terms} is equivalent to $\bm{Q}_\nabla=0$
if $\bm{Q}_\nabla$ is defined by \eqref{eq:Q-tractor}.
What remains is to check the well-definedness of $\bm{Q}_\nabla$.
It is by definition equivalent to that the
conformal transformation law of $Q_01+G_1\beta$ is as follows: if $\Hat{h}=e^{2\Upsilon}h$, then
\begin{equation}
	\label{eq:transform_of_bottom_component}
	\Hat{Q}_01+\Hat{G}_1\Hat{\beta}=e^{-n\Upsilon}(Q_01+G_1\beta-\braket{L_1\beta,d\Upsilon}).
\end{equation}
To show this, we recall from \cite{Aubry_Guillarmou_11}*{Corollary 4.14} that
the transformation laws of $Q_01$ and $G_1$ are $\Hat{Q}_01=e^{-n\Upsilon}(Q_01+nL_0\Upsilon)$ and
$\Hat{G}_1=e^{-n\Upsilon}(G_1-\iota_{\grad\Upsilon}L_1)$ (the first one is of course the well-known
transformation law of the $Q$-curvature).
We also note that $L_1$ vanishes on closed forms
and $L_0=(1/n)G_1d$ (see~\cite[Proposition 4.12]{Aubry_Guillarmou_11}). So we obtain
\begin{equation*}
	\begin{split}
		\Hat{G}_1\Hat{\beta}
		&=e^{-n\Upsilon}(G_1(\beta-d\Upsilon)-\braket{L_1(\beta-d\Upsilon),d\Upsilon})\\
		&=e^{-n\Upsilon}(G_1\beta-G_1d\Upsilon-\braket{L_1\beta,d\Upsilon})
		=e^{-n\Upsilon}(G_1\beta-nL_0\Upsilon-\braket{L_1\beta,d\Upsilon}).
	\end{split}
\end{equation*}
Hence \eqref{eq:transform_of_bottom_component} follows, and the proof of Theorem~\ref{thm:smoothness_extension}
is completed.

\section{Explicit computation on conformally Einstein manifolds}
\label{sec:conformally_Einstein}

In this section, we prove the explicit formula \eqref{eq:explicit_formula_for_conformally_Einstein}
of the operator $L_1$ on a conformally Einstein manifold $(M,\mathcal{C})$.
The proof here follows the symmetric 2-tensor case carried out in~\cite{Matsumoto_13}.
While the argument in~\cite{Matsumoto_13} was given in terms of the Fefferman--Graham ambient metric,
the same idea can also be implemented by the Poincar\'e metric, which we adopt in this exposition.

Suppose first that $\mathcal{C}$ does not necessarily carry Einstein representatives.
Without losing generality, we may assume that $M$ is the boundary of
an $(n+1)$-dimensional smooth compact manifold-with-boundary $\overline{X}$.
Identify an open neighborhood $\mathcal{U}$ of $M\subset\overline{X}$ with $M\times[0,\varepsilon)$.
We fix a representative $h\in\mathcal{C}$ once and for all, and let
\begin{equation*}
	g=\frac{dx^2+h_x}{x^2}
\end{equation*}
be a Poincar\'e metric for which $h_0=h$ and $h_x$ has an expansion in even powers of $x$
(see~\cite{Fefferman_Graham_02}).

Recall that, in Proposition \ref{prop:harmonic_extension_of_1-forms},
we called a 1-form $\eta\in\Omega^1(X)$ is $O^+(x^m)$ when $\eta$ is $O(x^m)$
and $\eta(\partial_x)=O(x^{m+1})$.
We now introduce some subspaces of such 1-forms.
For each even integer $w\ge -n+2$, let $\mathcal{A}[w]\subset\Omega^1(X)$ be the space of
1-forms that are expressed, near $\bdry X$, as
\begin{equation*}
	\eta=x^{-w}\beta_x+x^{-w+2}\varphi_x\frac{dx}{x},
\end{equation*}
where $\beta_x$ and $\varphi_x$ are smooth families of 1-forms and functions on $M$ in $x\in[0,\varepsilon)$
with expansions in even powers of $x$.
Moreover, we say that $\eta\in\mathcal{A}[w]$ is in $\mathcal{A}_\mathrm{df}[w]$ when $d_g^*\eta=O(x^n)$.
Note that $\mathcal{A}_\mathrm{df}[-n+2]=\mathcal{A}[-n+2]$ (use \eqref{eq:divergence_in_components} below).
For all $w\le -n$, we set $\mathcal{A}_\mathrm{df}[w]$ ($=\mathcal{A}[w]$) to be
\begin{equation*}
	\Set{\eta=x^{n-2}\beta_x+x^n\varphi_x\frac{dx}{x}|
		\text{$\beta_x$ and $\varphi_x$ are families as mentioned above such that $\beta_0=0$}}.
\end{equation*}
We need this somewhat irregular definition for technical reasons
which can be seen in the proof of Lemma \ref{lem:EFH}.
If $\eta\in\mathcal{A}[w]$, we call $\beta=\beta_0=(x^w\eta)|_{TM}\in\Omega^1(M)$ the \emph{restriction} of $\eta$,
and $\eta$ an \emph{extension} of $\beta$.
It is clear that the restriction of any element in $\mathcal{A}[w]$, $w\le -n$, is zero.

Consider the following three operators between these spaces:
\begin{alignat*}{2}
	E&\colon \mathcal{A}_\mathrm{df}[w]\longrightarrow\mathcal{A}_\mathrm{df}[w+2],&\qquad
	\eta&\longmapsto -\tfrac{1}{4}\eta,\\
	F&\colon \mathcal{A}_\mathrm{df}[w]\longrightarrow\mathcal{A}_\mathrm{df}[w-2],&\qquad
	\eta&\longmapsto (\Delta_g+w(w+n-2))\eta,\\
	H&\colon \mathcal{A}_\mathrm{df}[w]\longrightarrow\mathcal{A}_\mathrm{df}[w],&\qquad
	\eta&\longmapsto (w+n/2)\eta.
\end{alignat*}
We make the following observations on these operators.

\begin{lem}
	\label{lem:EFH}
	(1) The operators $E$, $F$, and $H$ above are well-defined and form an $\mathfrak{sl}_2$-triple.

	(2) Any $\beta\in\Omega^1(M)$ can be extended to some $\eta\in\mathcal{A}_\mathrm{df}[0]$.
\end{lem}

\begin{proof}
	The most nontrivial point about (1) is that $F$ maps $\mathcal{A}_\mathrm{df}[w]$ into
	$\mathcal{A}_\mathrm{df}[w-2]$.
	This can be checked using formulae of Aubry--Guillarmou~\cite{Aubry_Guillarmou_11}*{Equations (2.2), (2.3)}.
	Namely, if we decompose $\eta\in\mathcal{A}[w]$ into the tangential and normal parts as
	$\eta=\eta^{(t)}+\eta^{(n)}(dx/x)$, then
	\begin{equation}
		\label{eq:divergence_in_components}
		d_g^*\eta=
		\begin{pmatrix}
			x^2 d_{h_x}^* & -x\partial_x+n \\
			0 & 0
		\end{pmatrix}
		\begin{pmatrix}
			\eta^{(t)} \\ \eta^{(n)}
		\end{pmatrix}+O(x^{-w+4}),
	\end{equation}
	where the term indicated by $O(x^{-w+4})$ is expanded in even powers of $x$, and
	\begin{equation}
		\label{eq:Laplacian_in_components}
		\Delta_g\eta=
		\begin{pmatrix}
			-(x\partial_x)^2+(n-2)x\partial_x & 0 \\
			2x^2d_{h_x}^* & -(x\partial_x)^2+nx\partial_x
		\end{pmatrix}
		\begin{pmatrix}
			\eta^{(t)} \\ \eta^{(n)}
		\end{pmatrix}+\mathcal{A}[w-2].
	\end{equation}
	Here $\mathcal{A}[w-2]$ of course denotes some 1-form that belongs to this space.
	Let $\eta\in\mathcal{A}_\mathrm{df}[w]$.
	Then it is immediate from \eqref{eq:Laplacian_in_components} that
	$F\eta\in\mathcal{A}_\mathrm{df}[w-2]$ for $w\le -n+2$.
	For $w\ge -n+4$, observe first that the tangential part of $F\eta$ is $O(x^{-w+2})$.
	Since $d_g^*\eta=O(x^n)$, \eqref{eq:divergence_in_components} implies that
	$x^2d_{h_x}^*\eta^{(t)}+(w-2+n)\eta^{(n)}=O(x^{-w+4})$.
	Then a little computation shows that the normal part of $F\eta$ is $O(x^{-w+4})$.
	Hence $F\eta\in\mathcal{A}[w-2]$ also for $w\ge -n+4$.
	The fact that $d_g^*F\eta=O(x^n)$ is clear from \eqref{eq:Laplacian_on_functions} and
	$d_g^*F\eta=(\Delta_g+w(w+n-2))d_g^*\eta$.

	The assertion (2) follows easily from \eqref{eq:divergence_in_components}. Details are left to the reader.
\end{proof}

For our purpose, it is also important to note that an extension of $\beta$ in (2) can be constructed from
the harmonic extension $\overline{b}$ given in Proposition \ref{prop:harmonic_extension_of_1-forms}.
Using the notation there, we take
\begin{equation*}
	\eta=\beta+\sum_{k=1}^{n-3}x^k\beta_k+\left(\sum_{k=1}^{n-1}x^k\varphi_{k-1}\right)\frac{dx}{x}.
\end{equation*}
Then one can check that $\beta_k=0$ and $\varphi_{k-1}=0$ for $k$ odd, i.e., $\eta\in\mathcal{A}[0]$ actually.
Moreover,
\begin{equation}
	\label{eq:cutoff_term_of_approximate_harmonic_extension}
	\eta-\overline{b}=-x^{n-2}\log x\cdot L_1\beta
	-x^n\log x\cdot (G_1\beta)\frac{dx}{x}+O^+(x^{n-2})
\end{equation}
and $\overline{b}$ admits a polyhomogeneous expansion (see~\cite{Aubry_Guillarmou_11}).
Since $\overline{b}$ satisfies $d_g^*\overline{b}=0$ and it is known that $L_1\beta\in\im d_h^*$, we obtain
from \eqref{eq:cutoff_term_of_approximate_harmonic_extension} and
\eqref{eq:divergence_in_components} that $\eta\in\mathcal{A}_\mathrm{df}[0]$.
We will also need the fact that
\begin{equation}
	\label{eq:error_of_approximate_harmonic_extension}
	\Delta_g\eta=(n-2)x^{n-2}L_1\beta+nx^n(G_1\beta)\frac{dx}{x}+O^+(x^n),
\end{equation}
which follows from \eqref{eq:cutoff_term_of_approximate_harmonic_extension},
\eqref{eq:Laplacian_in_components}, and the fact that $G_1\beta\in\im d_h^*$.

\begin{rem}
	The three operators are also understood by the ambient metric.
	Recall from~\cite{Fefferman_Graham_02}*{Chapter 4} that the ambient metric is given as
	$\tilde{g}=s^2g-ds^2$ in the $(x,\xi,s)$-coordinates,
	which are related with the standard $(\rho,\xi,t)$-coordinates\footnote{It is even more standard to use $x$ for
	the coordinates on $M$, but we use $\xi$ instead as $x$ is already reserved.} on the ambient space
	$\tilde{\mathcal{G}}\cong\mathbb{R}\times M\times(0,\infty)$ by
	\begin{equation*}
		x=\sqrt{-2\rho},\qquad s=\sqrt{-2\rho}\,t
	\end{equation*}
	in the subdomain $\set{\rho<0}$.
	The Poincar\'e manifold $(X,g)$ can be seen as the hypersurface $\set{s=1}$ of $\tilde{\mathcal{G}}$.
	Let $\eta\in\mathcal{A}_\mathrm{df}[w]$, and for simplicity, assume that $w\ge -n+2$ and $d_g^*\eta=0$.
	Assign to it the 1-form $\tilde{\eta}=s^w\eta$ on $\set{\rho<0}\subset\tilde{\mathcal{G}}$.
	Then actually $\tilde{\eta}$ can be extended smoothly across $\rho=0$, and the restriction of $\eta$
	to $M$ corresponds to the pullback of $\tilde{\eta}$ to $\set{\rho=0, t=1}$.
	Now let $T=s\partial_s$. Then $E$, $F$, and $H$ correspond to
	\begin{equation*}
		\tilde{\eta}\longmapsto -\tfrac{1}{4}s^2\tilde{\eta},\qquad
		\tilde{\eta}\longmapsto \tilde{\Delta}\tilde{\eta},\quad\text{and}\quad
		\tilde{\eta}\longmapsto (\tilde{\nabla}_T+\tfrac{n}{2}+1)\tilde{\eta}.
	\end{equation*}
	For example, noting that $\tilde{\eta}(T)=0$,
	$\iota_T(d\tilde{\eta})=\mathcal{L}_T\tilde{\eta}=w\tilde{\eta}$,
	and the fact that $\tilde{g}=e^{2v}(g-dv^2)$ if we put $s=e^v$,
	by the conformal change law of the Hodge Laplacian we conclude that
	\begin{equation*}
		\Delta_{\tilde{g}}\tilde{\eta}
		=e^{-2v}(\Delta_{g-dv^{2}}\tilde{\eta}+(n-2)\iota_T(d\tilde{\eta}))
		=s^{w-2}(\Delta_g+w(w+n-2))\eta.
	\end{equation*}
	For general $\eta\in\mathcal{A}_\mathrm{df}[w]$, we need to introduce more careful assignment
	of ambient 1-forms. We omit it here. The case of $w\le -n$ is not important.
\end{rem}

We shall detect $L_1\beta$ in \eqref{eq:error_of_approximate_harmonic_extension} using the
commutation relations of $E$, $F$, and $H$
as in Graham--Jenne--Mason--Sparling~\cite{Graham_Jenne_Mason_Sparling_92}.
Note first that \eqref{eq:error_of_approximate_harmonic_extension} implies $F\eta=E^{n/2-2}\xi$
with some $\xi\in\mathcal{A}[-n+2]=\mathcal{A}_\mathrm{df}[-n+2]$ that restricts to
$(-4)^{n/2-2}(n-2)L_1\beta$. Then we can deduce that
\begin{equation*}
	\begin{split}
		F^{n/2-1}\eta=F^{n/2-2}E^{n/2-2}\xi
		&=(-1)^{n/2}(n/2-2)!H(H+1)\cdots(H+n/2-3)\xi+E\mathcal{A}_\mathrm{df}[-n]\\
		&=(n/2-2)!^2\xi+E\mathcal{A}_\mathrm{df}[-n].
	\end{split}
\end{equation*}
Let $\eta'\in\mathcal{A}_\mathrm{df}[0]$ be another extension of $\beta$.
Then since $\eta-\eta'\in E\mathcal{A}_\mathrm{df}[-2]$, it follows that
$F^{n/2-1}(\eta-\eta')\in E\mathcal{A}_\mathrm{df}[-n]$.
In particular, we can conclude that
\begin{equation}
	\label{eq:GJMS_construction}
	\begin{split}
		(\text{the restriction of $F^{n/2-1}\eta'$})
		&=(-4)^{n/2-2}(n-2)(n/2-2)!^2L_1\beta\\
		&=(-1)^{n/2}2^{n-3}(n/2-1)!(n/2-2)!L_1\beta
	\end{split}
\end{equation}
for \emph{any} extension $\eta'\in\mathcal{A}_\mathrm{df}[0]$ of $\beta$.

Now suppose there is an Einstein representative $h$ satisfying $\Ric(h)=2(n-1)\lambda h$
in the conformal class $\mathcal{C}$. In this case, one can take
$g=x^{-2}(dx^2+h_x)$, $h_x=(1-\frac{1}{2}\lambda x^2)^2h$ as the Poincar\'e metric.
Since $L_1$ annihilates the closed forms, by the de Rham--Hodge--Kodaira decomposition,
we may assume that $d_h^*\beta=0$ ($\beta\in\im d_h^*$ can even be assumed, but we do not need it here).
Because $h_x$ is conformal to $h$, we also have $d_{h_x}^*\beta=0$.
This implies that the pullback of $\beta$ by the projection
$M\times[0,\varepsilon)\longrightarrow M$ is a divergence-free extension of $\beta$.

We compute the Laplacian on 1-forms of the form $\psi(x)\alpha$, where $\alpha\in\Omega^1(M)$ is divergence-free.
By \eqref{eq:Laplacian_in_components}, $\Delta_g(\psi(x)\alpha)$ is again in this form and
\begin{equation*}
	\Delta_g(\psi(x)\alpha)
	=\left(-(x\partial_x)^2+(n-2)\frac{1-\frac{1}{2}\lambda x^2}{1+\frac{1}{2}\lambda x^2}x\partial_x\right)\psi(x)
	\alpha
	+\frac{x^2}{(1-\frac{1}{2}\lambda x^2)^2}\psi(x)\Delta_h\alpha.
\end{equation*}
If we put $y=x(1-\frac{1}{2}\lambda x^2)^{-1}$, then
\begin{equation*}
	\Delta_g(\psi(x)\alpha)
	=\left(-(y\partial_y)^2+(n-2)y\partial_y
	-2\lambda y^2(y\partial_y)^2+2(n-3)\lambda y^2\cdot y\partial_y\right)\psi\alpha
	+y^2\psi\Delta_h\alpha.
\end{equation*}
Hence, if we take $\psi(x)=y^w$, then $F(y^w\alpha)=y^{-w+2}(\Delta_h-2\lambda w(w-n+3))\alpha$.
By applying this repeatedly, we obtain
\begin{equation*}
	F^{n/2-1}\beta=y^{n-2}\left(\prod_{w=0}^{n/2-2}(\Delta_h-2\lambda w(w-n+3))\right)\beta,
\end{equation*}
which combined with \eqref{eq:GJMS_construction} gives the formula of $L_1\beta$ for divergence-free
1-forms $\beta$.
Reformulating it for general 1-forms, we get \eqref{eq:explicit_formula_for_conformally_Einstein}.

\bibliography{myrefs}

\end{document}